\RequirePackage{ifthen}
\RequirePackage{ifpdf}
\newcommand{\driverOption}{}
\ifthenelse{\boolean{pdf}}{
  \renewcommand{\driverOption}{pdftex}
} { 
  \renewcommand{\driverOption}{dvips}
}

\documentclass[reqno, 11pt, letterpaper, commented, oneside, \driverOption]{amsart}

\newboolean{isCommented}
\usepackage{geometry}
\usepackage{bbm}
\usepackage[final]{graphicx}
\usepackage{upref}
\usepackage{enumerate}
\usepackage{latexsym}
\usepackage{amssymb}
\usepackage[ansinew]{inputenc}
\usepackage[T1]{fontenc}
\usepackage{mathtools}
\usepackage{mycommands}

\ifthenelse{\boolean{pdf}}{
  \usepackage[final]{ps4pdf}
} { 
  \usepackage[inactive]{ps4pdf}
}
\PSforPDF{\usepackage{psfrag}}

\newcommand{\hyperrefDriverOption}{}
\ifthenelse{\boolean{pdf}}{
	\renewcommand{\hyperrefDriverOption}{pdftex}
} { 
	\renewcommand{\hyperrefDriverOption}{hypertex}
}
\usepackage[\hyperrefDriverOption,
  colorlinks = false, 
  pdfauthor={Torsten Mütze, Pascal Su},
  pdftitle={Bipartite Kneser graphs are Hamiltonian}]
  {hyperref}
    
\ifthenelse{\boolean{isCommented}} {
	\newcommand{\TM}[1]{\marginpar{\parbox{4cm}{{\small {\bf TM:} #1}}}}
	\newcommand{\FW}[1]{\marginpar{\parbox{4cm}{{\small {\bf FW:} #1}}}}
} { 
	\newcommand{\TM}[1]{}
	\newcommand{\FW}[1]{}
}

\newtheorem{theorem}{Theorem}
\newtheorem{lemma}[theorem]{Lemma}

\newtheorem{conjecture}[theorem]{Conjecture}

\theoremstyle{definition}

\theoremstyle{remark}

\long\def\symbolfootnote[#1]#2{\begingroup
\def\thefootnote{\fnsymbol{footnote}}\footnote[#1]{#2}\endgroup}

\ifthenelse{\boolean{isCommented}} {
	\geometry{
	  hmargin={25mm, 50mm},
	  marginparwidth=40mm, 
	  vmargin={25mm, 25mm},
	  headsep=10mm,
	  headheight=5mm,
	  footskip=10mm
	}
} { 
	\geometry{
	  hmargin={25mm, 25mm}, 
	  vmargin={25mm, 25mm},
	  headsep=10mm,
	  headheight=5mm,
	  footskip=10mm
	}
}

\begin{document}

\begin{center}

\LARGE Bipartite Kneser graphs are Hamiltonian\footnote{An extended abstract of this work has appeared in the proceedings of the European Conference on Combinatorics, Graph Theory and Applications (Eurocomb) 2015.}
\vspace{2mm}

\Large Torsten Mütze\footnotemark[2], Pascal Su
\footnotetext[2]{The author was supported by a fellowship of the Swiss National Science Foundation. This work was completed when the author was with the School of Mathematics at Georgia Institute of Technology, 30332 Atlanta GA, USA.}%
\vspace{2mm}

\large
  Department of Computer Science \\
  ETH Zürich, 8092 Zürich, Switzerland \\
  {\small\tt torsten.muetze@inf.ethz.ch, sup@student.ethz.ch}
\vspace{5mm}

\small

\begin{minipage}{0.8\linewidth}
\textsc{Abstract.}
For integers $k\geq 1$ and $n\geq 2k+1$ the Kneser graph $K(n,k)$ has as vertices all $k$-element subsets of $[n]:=\{1,2,\ldots,n\}$ and an edge between any two vertices (=sets) that are disjoint. The bipartite Kneser graph $H(n,k)$ has as vertices all $k$-element and $(n-k)$-element subsets of $[n]$ and an edge between any two vertices where one is a subset of the other. It has long been conjectured that all Kneser graphs and bipartite Kneser graphs except the Petersen graph $K(5,2)$ have a Hamilton cycle. The main contribution of this paper is proving this conjecture for bipartite Kneser graphs $H(n,k)$. We also establish the existence of cycles that visit almost all vertices in Kneser graphs $K(n,k)$ when $n=2k+o(k)$, generalizing and improving upon previous results on this problem.
\end{minipage}

\vspace{2mm}

\begin{minipage}{0.8\linewidth}
\end{minipage}



\end{center}

\vspace{3mm}


\section{Introduction}

The question whether a graph has a Hamilton cycle --- a cycle that visits every vertex exactly once --- is a fundamental graph theoretical problem with a wide range of practical applications, shown to be NP-complete already in Karp's landmark paper \cite{Karp72}. As a consequence, recent years have seen an increasing interest in Hamiltonicity problems in various different flavors and the solution of several long-standing open problems (the survey \cite{kuehn:osthus:survey} gives an excellent overview of these developments).

\subsection{Hamilton cycles in (bipartite) Kneser graphs}

The question whether a graph has a Hamilton cycle turns out to be surprisingly difficult even for families of graphs defined by very simple algebraic constructions. Two prominent examples of this phenomenon are the Kneser graph and the bipartite Kneser graph (Kneser graphs were introduced by Lov{\'a}sz in his celebrated proof of Kneser's conjecture \cite{MR514625}). For integers $n$ and $k$ satisfying $k\geq 1$ and $n\geq 2k+1$, the \emph{Kneser graph} $K(n,k)$ has as vertices all $k$-element subsets of $[n]:=\{1,2,\ldots,n\}$, and an edge between any two vertices (=sets) that are disjoint. The \emph{bipartite Kneser graph} $H(n,k)$ has as vertices all $k$-element and all $(n-k)$-element subsets of $[n]$, and an edge between any two vertices where one is a subset of the other.
The Kneser graphs and bipartite Kneser graphs have long been conjectured to have a Hamilton cycle, apart from one notorious exception, namely the Petersen graph $K(5,2)$:

\begin{conjecture}
\label{conj:kneser}
For any $k\geq 1$ and $n\geq 2k+1$, except for $(n,k)=(5,2)$, the Kneser graph $K(n,k)$ has a Hamilton cycle.
\end{conjecture}

\begin{conjecture}
\label{conj:bip-kneser}
For any $k\geq 1$ and $n\geq 2k+1$, the bipartite Kneser graph $H(n,k)$ has a Hamilton cycle.
\end{conjecture}

In the numerous papers on the subject (see below), the sparsest among these graphs, the so-called \emph{odd graph} $K(2k+1,k)$ and the \emph{middle layer graph} $H(2k+1,k)$ have received particular attention, as proving Hamiltonicity for the sparsest graphs is particularly intricate:

\begin{conjecture}
\label{conj:odd}
For any $k\geq 1$, except for $k=2$, the odd graph $K(2k+1,k)$ has a Hamilton cycle.
\end{conjecture}

\begin{conjecture}
\label{conj:middle-levels}
For any $k\geq 1$, the middle layer graph $H(2k+1,k)$ has a Hamilton cycle.
\end{conjecture}

One of the main motivations for these conjectures is a classical and vastly more general conjecture due to Lov{\'a}sz~\cite{MR0263646}, which asserts that, apart from five exceptional graphs (one of the exceptions $K(5,2)$ we already mentioned), every connected vertex-transitive graph has a Hamilton cycle. A vertex-transitive graph is a graph that `looks the same' from the point of view of any vertex, and Kneser graphs and bipartite Kneser graphs have this strong symmetry property (and they are connected for the given range of parameters), so these conjectures represent a highly nontrivial special case of Lov{\'a}sz' conjecture.

\subsection{Known results}

Conjecture~\ref{conj:odd} was raised by Meredith and Lloyd \cite{MR0457282} (see also \cite{MR556008}).
In a sequence of papers \cite{MR0457282,MR0389663,MR510592,MR888679,MR1778200,MR1883565,MR2836824}, the conjecture and its generalization, Conjecture~\ref{conj:kneser}, were verified for ever increasing ranges of parameters.
To date, Conjecture~\ref{conj:kneser} has been confirmed with the help of computers for all $n\leq 27$ and all relevant values of $k$ \cite{MR2020936}, and the best known general result is due to Chen:

\begin{theorem}[\cite{MR1999733}]
\label{thm:chen-kneser}
For any $k\geq 1$ and $n\geq 2.62k+1$, the Kneser graph $K(n,k)$ has a Hamilton cycle.
\end{theorem}

As an important step towards settling Conjecture~\ref{conj:odd}, Johnson showed that the odd graph contains a cycle that visits almost all vertices:

\begin{theorem}[\cite{MR2046083}]
\label{thm:johnson-odd}
There exists a constant $c$, such that for any $k\geq 1$, the odd graph $K(2k+1,k)$ has a cycle that visits at least a $(1-\frac{c}{\sqrt{k}})$-fraction of all vertices.
\end{theorem}


Conjecture~\ref{conj:bip-kneser} was raised independently by Simpson~\cite{MR1152123} and Roth (see \cite{gould:91} and \cite{MR1271867}). Since then, there has been steady progress on the problem \cite{MR1282567,MR1271867,MR1778200}, and similarly to before, the conjecture has been confirmed for all $n\leq 27$ and all relevant values of $k$ \cite{MR2020936}, and the best known general result is due to Chen:

\begin{theorem}[\cite{MR1999733}]
\label{thm:chen-bip-kneser}
For any $k\geq 1$ and $n\geq 2.62k+1$, the bipartite Kneser graph $H(n,k)$ has a Hamilton cycle.
\end{theorem}


Conjecture~\ref{conj:middle-levels}, also known as the \emph{middle levels conjecture} or \emph{revolving door conjecture}, originated probably with Havel~\cite{MR737021} and Buck and Wiedemann~\cite{MR737262}, but has also been attributed to Dejter, Erd{\H{o}}s, Trotter~\cite{MR962224} and various others.
This conjecture has attracted considerable attention over the years \cite{MR2548541,shimada-amano,savage:93,MR1350586,MR1329390,MR2046083,MR962223,MR962224,MR1268348,horakEtAl:05,Gregor20102448}, and a proof of it has only been announced very recently.

\begin{theorem}[\cite{muetze:14}]
\label{thm:middle-levels}
For any $k\geq 1$, the middle layer graph $H(2k+1,k)$ has a Hamilton cycle.
\end{theorem}

\subsection{Hamilton cycles in the hypercube}

The main reason for the interest in the middle levels conjecture is its relation to the hypercube graph and to Gray codes, two themes of fundamental interest for combinatorialists (see the surveys \cite{MR949280} and \cite{Savage:1997}, respectively).
The \emph{hypercube} $Q(n)$ is the graph which has as vertices all bitstrings of length $n$, and an edge between any two bitstrings that differ in exactly one bit. Partitioning the vertices of $Q(n)$ into \emph{levels} $0,\ldots,n$ according to the number of 1-entries in the bitstrings, and denoting by $Q(n,k)$ the subgraph of $Q(n)$ induced by all vertices in level $k$ and $k+1$, it is easy to see that $H(2k+1,k)$ and $Q(2k+1,k)$ are isomorphic. So the middle levels conjecture asserts that the subgraph $Q(2k+1,k)$ of the cube has a Hamilton cycle.
Observe that Hamilton cycles in the cube or subgraphs of it correspond to certain \emph{Gray codes}, i.e., cyclic sequences of binary code words with the property that any two consecutive code words differ in exactly one bit.
Clearly, $Q(2k+1,k)$ is the only subgraph of the cube induced by two consecutive levels that have the same size, and where we can hope to find a Hamilton cycle. Nevertheless, the following is a natural generalization of the middle levels conjecture (in a different direction than Conjecture~\ref{conj:bip-kneser}, cf.\ also \cite{Gregor20102448}), which provides a nice structural insight about the cube and establishes the existence of various additional families of restricted Gray codes:

\begin{theorem}
\label{thm:cube}
For any $n\geq 3$ and $k\in\{1,2,\ldots,n-2\}$, the graph $Q(n,k)$ has a cycle that visits all vertices in the smaller of the levels $k$ and $k+1$.
\end{theorem}

It was already noted in \cite{MR737021} that with a simple inductive construction, Theorem~\ref{thm:cube} can be derived easily from Theorem~\ref{thm:middle-levels}. In fact, the results of this paper will be proved using a further refinement of this proof technique.

\section{Our results}

The main contribution of this paper is a proof of Conjecture~\ref{conj:bip-kneser}.

\begin{theorem}
\label{thm:bip-kneser}
For any $k\geq 1$ and $n\geq 2k+1$, the bipartite Kneser graph $H(n,k)$ has a Hamilton cycle.
\end{theorem}

We also make some progress towards Conjecture~\ref{conj:kneser} (and the special case Conjecture~\ref{conj:odd}), by generalizing and improving Theorem~\ref{thm:johnson-odd} as follows:

\begin{theorem}
\label{thm:odd}
For any $k\geq 1$ and $n\geq 2k+1$, the Kneser graph $K(n,k)$ has a cycle that visits at least a $\frac{2k}{n}$-fraction of all vertices.
In particular, for any $k\geq 1$, the odd graph $K(2k+1,k)$ has a cycle that visits at least a $(1-\frac{1}{2k+1})$-fraction of all vertices.
\end{theorem}

Note that the cycle guaranteed by Theorem~\ref{thm:odd} visits almost all vertices of $K(n,k)$, i.e., a $(1-o(1))$-fraction, whenever $n=2k+o(k)$.

\section{Key lemma and proof of theorems}

Our results are immediate consequences of the following lemma, illustrated in Figure~\ref{fig:lemma} below.
This lemma therefore represents a powerful `bootstrapping' method that extends Theorem~\ref{thm:middle-levels} to a large range of other interesting graphs.
To state the lemma, we say that a path in the hypercube $Q(n)$ is \emph{monotone}, if it visits at most one vertex in every level.

\begin{lemma}
\label{lem:main}
For any $k\geq 1$ and $n\geq 2k+1$, there is a cycle $C(n,k)$ in the graph $Q(n,k)\seq Q(n)$ that visits all $\binom{n}{k}$ vertices in level $k$, and a set of $\binom{n}{k}$ vertex-disjoint monotone paths $\cP(n,k)$ in $Q(n)$, each of which starts at a vertex of the cycle $C(n,k)$ in level $k+1$ and ends at a vertex in level $n-k$.
\end{lemma}

\begin{figure}[ht!]
\centering
\PSforPDF{
 \psfrag{qn}{$Q(n)$}
 \psfrag{cnk}{$C(n,k)\seq Q(n,k)$}
 \psfrag{pnk}{$\cP(n,k)$}
 \psfrag{k}{$k$}
 \psfrag{kp1}{$k+1$}
 \psfrag{nmk}{$n-k$}
 \includegraphics{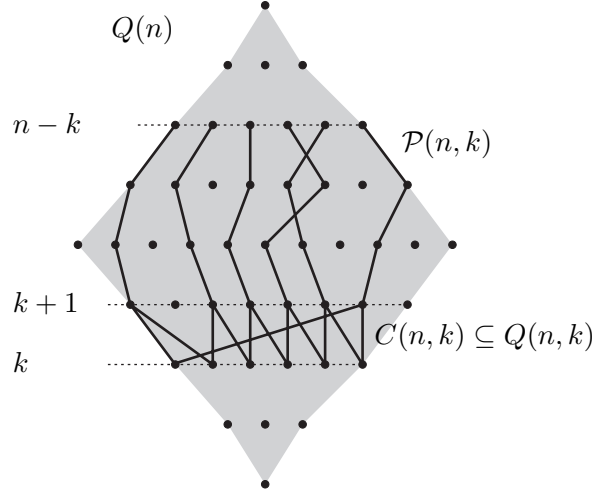}
}
\caption{Illustration of Lemma~\ref{lem:main}.}
\label{fig:lemma}
\end{figure}

Note that the conditions of Lemma~\ref{lem:main} enforce that each vertex of the cycle $C(n,k)$ in level $k+1$ of $Q(n)$ is contained in exactly one of the paths from $\cP(n,k)$. Observe also that the paths $\cP(n,k)$ must visit all vertices in level $n-k$ of $Q(n)$ (there are only $\binom{n}{n-k}=\binom{n}{k}$ such vertices), but leave out some vertices in levels $k+1,k+2,\ldots,n-k-1$.

Furthermore, Lemma~\ref{lem:main} is a strengthening of Theorem~\ref{thm:cube} (for the theorem, the paths $\cP(n,k)$ are ignored, and the cycle $C(n,k)$ alone has the desired properties). The proof of the lemma is a relatively straightforward induction, making essential use of Theorem~\ref{thm:middle-levels}.
We defer the proof to the next section.

With Lemma~\ref{lem:main} in hand, proving Theorems~\ref{thm:bip-kneser} and \ref{thm:odd} is easy.

\begin{proof}[Proof of Theorem~\ref{thm:bip-kneser}]
Let $n$ and $k$ be as in the theorem, and let $C(n,k)$ and $\cP(n,k)$ be the cycle and the set of paths given by Lemma~\ref{lem:main}.
The cycle $C(n,k)$ visits all $N:=\binom{n}{k}$ vertices in level $k$, and it has the form $(x_1,x_2,\ldots,x_{2N})$, where the $x_{2i-1}$ and the $x_{2i}$, $i=1,\ldots,N$, are vertices in level $k$ and level $k+1$, respectively. Moreover, every $x_{2i}$ is obtained from $x_{2i-1}$ or from $x_{2i+1}$ (indices are considered modulo $2N$) by flipping a single 0-bit to a 1-bit. For $i=1,\ldots,N$ consider the path from $\cP(n,k)$ whose first vertex is $x_{2i}$, and let $y_{2i}$ be its end vertex in level $n-k$. As the path is monotone, $y_{2i}$ is obtained from $x_{2i}$ by flipping $(n-k)-(k+1)=n-2k-1$ many 0-bits to 1-bits.
Now consider the cyclic sequence $(x_1,y_2,x_3,y_4,x_5,y_6,\ldots,x_{2N-1},y_{2N})$ of vertices.
Note that the vertices $\{x_{2i-1}\mid i=1,\ldots,N\}$, are all vertices in level $k$, the vertices $\{y_{2i}\mid i=1,\ldots,N\}$ are all vertices in level $n-k$ (the paths from $\cP(n,k)$ are vertex-disjoint). Moreover, every $y_{2i}$ is obtained from $x_{2i-1}$ or from $x_{2i+1}$ by flipping $n-2k$ many 0-bits to 1-bits. Interpreting the bitstrings in this sequence as characteristic vectors of subsets of $[n]$, we thus obtain the desired Hamilton cycle in $H(n,k)$.
\end{proof}

\begin{proof}[Proof of Theorem~\ref{thm:odd}]
For $k=1$ and $n\geq 3$ the graph $K(n,1)$ is the complete graph on $n$ vertices and trivially has a Hamilton cycle.
So let $k\geq 2$ and $n\geq 2k+1$, and let $C(n-1,k-1)$ and $\cP(n-1,k-1)$ be the cycle and the set of paths given by Lemma~\ref{lem:main}. The paths in $\cP(n-1,k-1)$ start in level $k$ and end in level $(n-1)-(k-1)=n-k$, and therefore have length $n-2k\geq 1$.
The cycle $C(n-1,k-1)$ visits all $N:=\binom{n-1}{k-1}$ vertices in level $k-1$, and it has the form $(x_1,x_2,\ldots,x_{2N})$, where the $x_{2i-1}$ and the $x_{2i}$, $i=1,\ldots,N$, are vertices in level $k-1$ and level $k$, respectively. Moreover, every $x_{2i}$ is obtained from $x_{2i-1}$ or from $x_{2i+1}$ (indices are considered modulo $2N$) by flipping a single 0-bit to a 1-bit. For $i=1,\ldots,N$ consider the path from $\cP(n-1,k-1)$ whose first vertex is $x_{2i}$, and let $y_{2i}$ be the vertex of this path in level $n-k-1$ (the end vertex of this path is on the next higher level $n-k$). As the path is monotone, $y_{2i}$ is obtained from $x_{2i}$ by flipping $(n-k-1)-k=n-2k-1$ many 0-bits to 1-bits.
For $i=1,\ldots,N$, let $x_{2i-1}^+$ be the bitstring obtained from $x_{2i-1}$ by adding an additional 1-bit, and let $\ol{y_{2i}}^+$ be the bitstring obtained from $y_{2i}$ by inverting all bits and adding an additional 0-bit. Note that $x_{2i-1}^+$ and $\ol{y_{2i}}^+$ both have length $n$ and contain exactly $k$ entries equal to 1.
Now consider the cyclic sequence of vertices $(x_1^+,\ol{y_2}^+,x_3^+,\ol{y_4}^+,x_5^+,\ol{y_6}^+,\ldots,x_{2N-1}^+,\ol{y_{2N}}^+)$.
Note that all vertices in this sequence are different (here we use that the $y_{2i}$ are all different, as the paths from $\cP(n-1,k-1)$ are vertex-disjoint).
Moreover, for every $\ol{y_{2i}}^+$ we have that at each position with a 1-bit, both $x_{2i-1}^+$ and $x_{2i+1}^+$ have a 0-bit.
Interpreting the bitstrings in this sequence as characteristic vectors of subsets of $[n]$, we thus obtain a cycle of length $2N=2\binom{n-1}{k-1}$ in $K(n,k)$. The total number of vertices of $K(n,k)$ is $\binom{n}{k}$, so the fraction of vertices visited by the cycle is $2\binom{n-1}{k-1}/\binom{n}{k}=\frac{2k}{n}$.
\end{proof}


\section{Proof of Lemma~\texorpdfstring{\ref{lem:main}}{12}}

Lemma~\ref{lem:main} is an immediate consequence of the following lemma, which slightly strengthens the conditions on the cycle $C(n,k)$ and the paths $\cP(n,k)$ by enforcing and forbidding certain vertices to be visited.
To state the lemma and the proof, we introduce a bit of notation:
For bitstrings $x$ and $y$ we use $x\circ y$ to denote the concatenation of $x$ and $y$. Moreover, for any graph $G$ whose vertices are bitstrings and any bitstring $y$ we denote by $G\circ y$ the graph obtained from $G$ by replacing every vertex $x$ by $x\circ y$.
For integers $n$ and $k$ satisfying $n\geq 1$ and $0\leq k\leq n$ we denote by $a(n,k)$ the bitstring of length $n$ that has $k$ many 1-bits at the last $k$ positions (and $n-k$ leading 0-bits). Moreover, for integers $n$ and $k$ satisfying $n\geq 2$ and $1\leq k\leq n-1$ we define $b(n,k):=a(n-1,k)\circ 0$. Note that $a(n,k)$ and $b(n,k)$ are different vertices of $Q(n)$ in level $k$.

\begin{lemma}
\label{lem:aux}
For any $k\geq 1$ and $n\geq 2k+1$, there is a cycle $C(n,k)$ in the graph $Q(n,k)\seq Q(n)$ that visits all $\binom{n}{k}$ vertices in level $k$, and a set of $\binom{n}{k}$ vertex-disjoint monotone paths $\cP(n,k)$ in $Q(n)$, each of which starts at a vertex of the cycle $C(n,k)$ in level $k+1$ and ends at a vertex in level $n-k$, with the following additional properties:
\begin{enumerate}[(i)]
\item The cycle $C(n,k)$ contains the path $D(n,k):=\big(a(n,k),a(n,k+1),b(n,k)\big)$.
\item The path from $\cP(n,k)$ that starts at the vertex $a(n,k+1)$ is given by $A(n,k):=\big(a(n,k+1),a(n,k+2),\ldots,a(n,n-k)\big)$.
\item None of the paths from $\cP(n,k)$ has a vertex in common with the path $B(n,k):=\big(b(n,k+1),b(n,k+2),\ldots,b(n,n-k-1)\big)$.
\end{enumerate}
\end{lemma}

The path $B(n,k)$ in condition~(iii) is another monotone path in $Q(n)$ different from the ones in $\cP(n,k)$. It starts in level $k+1$ and ends in level $n-k-1$ (so it ends one level below the paths from $\cP(n,k)$). If $n=2k+1$, then $B(n,k)=\emptyset$, and then condition~(iii) is trivially satisfied.
Note that condition~(iii) implies that the cycle $C(n,k)$ does not visit the vertex $b(n,k+1)$ in level $k+1$.

The proof of Lemma~\ref{lem:aux} is split into three parts (illustrated in Figure~\ref{fig:proof}).
The main part of the proof (part (c)) is a relatively straightforward induction, which constructs the cycle $C(n,k)$ and the corresponding paths $\cP(n,k)$ from $C(n-1,k)$, $\cP(n-1,k)$ and from $C(n-1,k-1)$ and $\cP(n-1,k-1)$. The other two cases are the base cases of the induction. One base case $(k,n)=(1,n)$ (part (b)) is easily verified `manually', and the other base case $(k,n)=(k,2k+1)$ (part (a)) is exactly the middle levels conjecture, which we know to be true by Theorem~\ref{thm:middle-levels}.

\begin{figure}
\centering
\PSforPDF{
 \psfrag{n}{$n$}
 \psfrag{k}{$k$}
 \psfrag{a}{(a)}
 \psfrag{b}{(b)}
 \psfrag{c}{(c)}
 \psfrag{1}{1}
 \psfrag{2}{2}
 \psfrag{3}{3}
 \psfrag{4}{4}
 \psfrag{5}{5}
 \psfrag{6}{6}
 \psfrag{7}{7}
 \psfrag{8}{8}
 \psfrag{9}{9}
 \psfrag{10}{10}
 \includegraphics{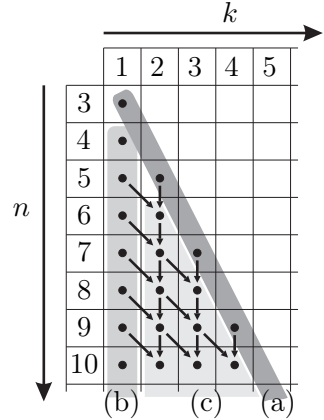}
}
\caption{Structure of the proof of Lemma~\ref{lem:aux}. Every bullet represents a pair $(n,k)$ for which the lemma must be verified. The grey regions indicate which pairs of values are dealt with in which of the three parts (a), (b) and (c) of the proof. The arrows illustrate the induction step in part~(c).}
\label{fig:proof}
\end{figure}

\begin{proof}[Proof of Lemma~\ref{lem:aux}, part (a): the case $k\geq 1$, $n=2k+1$]
Consider any Hamilton cycle in the graph $H(2k+1,k)$ given by Theorem~\ref{thm:middle-levels}. Fix any three consecutive vertices in level $k$, $k+1$ and $k$ on this cycle. By applying a suitable bit permutation, these three vertices can be mapped onto the three vertices $a(n,k)$, $a(n,k+1)$ and $b(n,k)$ required by condition~(i). As permuting bits is an automorphism of the graph $H(2k+1,k)$, we obtain a cycle $C(2k+1,k)$ satisfying condition~(i). For the sets of paths $\cP(2k+1,k)$ we take all $\binom{2k+1}{k+1}=\binom{2k+1}{k}$ vertices in level $k+1$ (each path consists only of a single vertex). These paths satisfy all requirements of the lemma. In particular, conditions~(ii) and (iii) are trivially true.
\end{proof}

\begin{proof}[Proof of Lemma~\ref{lem:aux}, part (b): the case $k=1$, $n\geq 4$]
We construct a cycle $C(n,1)$ and a set of paths $\cP(n,1)$ satisfying the requirements of the lemma in two steps. In the first step we define an auxiliary cycle $C'(n,1)$ and an auxiliary set of paths $\cP'(n,1)$ that satisfy all conditions except (iii). In the second step we transform these auxiliary subgraphs by permuting two bits (which is an automorphism of the cube), with the effect that condition (iii) is met as well.

For any bitstring $x$ and any integer $\ell$ we define $\sigma^\ell(x)$ as the bitstring obtained from $x$ by cyclically shifting it left by $\ell$ positions. Moreover, for any graph $G$ whose vertices are bitstrings we let $\sigma^\ell(G)$ denote the graph obtained from $G$ by replacing each vertex $x$ by $\sigma^\ell(x)$.
Let $C'(n,1)$ be the cycle obtained as the union of the paths $\sigma^\ell(D(n,1))$, $\ell=0,1,\ldots,n-1$, where $D(n,1)$ is defined as in condition~(i) of the lemma (note that $b(n,k)=\sigma^1(a(n,k))$). This cycle visits all $\binom{n}{1}=n$ vertices in level 1 and clearly satisfies condition~(i).
Moreover, let $\cP'(n,1)$ be the union of the $\binom{n}{1}=n$ paths $\sigma^\ell(A(n,1))$, $\ell=0,1,\ldots,n-1$, where $A(n,1)$ is defined as in condition~(ii) of the lemma. Clearly, these paths are vertex-disjoint and monotone, each of them starts at a vertex of the cycle $C'(n,1)$ in level 2 and ends at a vertex in level $n-1$, and condition~(ii) is satisfied.
However, the paths $\cP'(n,1)$ violate condition~(iii). In fact, the path $\sigma^1(A(n,1))$ properly contains the `forbidden' path $B(n,1)$.

Let $C(n,1)$ and $\cP(n,1)$ be the subgraphs of the cube $Q(n)$ obtained from $C'(n,1)$ and $\cP'(n,1)$ by permuting the last two bits. As permuting bits is an automorphism of the cube, and as the path $D(n,1)\seq C'(n,1)$ and the path $A(n,1)$ are invariant under this permutation, the resulting cycle $C(n,1)$ and the set of paths $\cP(n,1)$ satisfy all requirements of the lemma, in particular conditions~(i) and (ii). To verify condition~(iii), observe that any vertex on the path $B(n,1)$ is a bitstring whose first bit is 0 and whose last three bits are (1,1,0). It follows that the preimage of $B(n,1)$ when permuting the last two bits is a path which has the property that all of its vertices are bitstrings whose 1-entries do \emph{not} appear consecutively, even when viewing them as cyclic bitstrings.
However, the vertices visited by the paths $\cP'(n,1)$ all have the property that their 1-entries appear consecutively when viewing them as cyclic bitstrings. We conclude that none of the paths in $\cP(n,1)$ has a vertex in common with the path $B(n,1)$. This completes the proof.
\end{proof}

\begin{proof}[Proof of Lemma~\ref{lem:aux}, part (c): the case $k\geq 2$, $n\geq 2k+2$]
For the reader's convenience, the notations used in this proof are illustrated in Figure~\ref{fig:ind}.
We prove this part by induction over $n$, assuming that the lemma holds for $n-1$ and all corresponding values of $k$. The cases (a) and (b) of the lemma proved before serve as our induction basis (see Figure~\ref{fig:proof}).
For the induction step let $k\geq 2$ and $n\geq 2k+2$ be fixed. We consider the decomposition of $Q(n)$ into $Q(n-1)\circ 0$, $Q(n-1)\circ 1$ and the perfect matching $M(n):=\{(x\circ 0,x\circ 1)\mid x\in\{0,1\}^{n-1}\}$. In other words, the vertices of $Q(n)$ are partitioned according to the value of the last bit, yielding two copies of $Q(n-1)$ plus the matching $M(n)$, which is formed by the edges of $Q(n)$ along which the last bit is flipped (see Figure~\ref{fig:ind}).
By induction, there are subgraphs $C(n-1,k)$, $\cP(n-1,k)$ and $C(n-1,k-1)$, $\cP(n-1,k-1)$ of $Q(n-1)$ satisfying the conditions of the lemma.

Let $C_0^-$ be the path obtained by removing from $C(n-1,k)$ both edges from $D(n-1,k)$ and the middle vertex $a(n-1,k+1)$.
Let $C_1^-$ be the path obtained by replacing in $C(n-1,k-1)$ the edge $(b(n-1,k-1),a(n-1,k))$ (this is the second edge of $D(n-1,k-1)$) by the edge $(b(n-1,k-1),b(n-1,k))$ (the vertex $b(n-1,k)$ is the first vertex on the path $B(n-1,k-1)$ and hence not contained in $C(n-1,k-1)$ by condition~(iii)).
Now let $C(n,k)$ be the cycle obtained as the union of $C_0^-\circ 0$, $C_1^-\circ 1$ plus the two edges $(a(n-1,k)\circ 0,a(n-1,k)\circ 1)=(b(n,k),a(n,k+1))$ and $(b(n-1,k)\circ 0,b(n-1,k)\circ 1)$ from the matching $M(n)$ (see Figure~\ref{fig:ind}). It is easy to check that $C(n,k)$ visits all $\binom{n-1}{k}+\binom{n-1}{k-1}=\binom{n}{k}$ vertices in level $k$ of $Q(n)$ and that it contains the path $D(n,k)$ defined in condition~(i): The second edge of $D(n,k)$ is given by the first of the two edges from $M(n)$ added before, the first edge of $D(n,k)$ is given by attaching a 1-bit to the first edge of $D(n-1,k-1)$ which is contained in $C_1^-$.

We partition the set of end vertices of the paths in $\cP(n-1,k)$ in level $n-k-1$ except the two vertices $a(n-1,n-k-1)$ and $b(n-1,n-k-1)$ into two sets $X$ and $Y$ as follows: the set $X$ consists of all vertices that are contained in one of the paths from $\cP(n-1,k-1)$, and the set $Y$ consists of all vertices that are not contained in any of the paths from $\cP(n-1,k-1)$.
Let $E_X$ denote the set of edges from the paths in $\cP(n-1,k-1)$ between level $n-k-1$ and $n-k$ that have one vertex in the set $X$ (these are the terminal edges of these paths). We claim that no edge in $E_X$ has $b(n-1,n-k-1)$ or $a(n-1,n-k)$ as its end vertex: For the vertex $b(n-1,n-k-1)$ this follows directly from the definition of $X$. For the vertex $a(n-1,n-k)$ this follows since the edge $(a(n-1,n-k-1),a(n-1,n-k))$ is contained in the path $A(n-1,k-1)$ and $a(n-1,n-k-1)$ is not part of $X$ by definition.

Let $\cP_0$ be the paths obtained from $\cP(n-1,k)\setminus\{A(n-1,k)\}$ by extending the paths that have an end vertex in $X$ by the edges $E_X$ and by extending the path that ends at the vertex $b(n-1,n-k-1)$ by the edge $(b(n-1,n-k-1),a(n-1,n-k))$ (by adding this edge the paths remain vertex-disjoint by our previous observation about the edges $E_X$).
Clearly, all paths in $\cP_0$ except the ones whose end vertex is in $Y$ end in level $n-k$ of $Q(n-1)$.

Let $\cP_1$ be the paths obtained from the paths $\cP(n-1,k-1)$ by removing the vertices in level $n-k$ and by adding the path $B(n-1,k-1)$. By condition~(iii) this yields a set of vertex-disjoint paths that end in level $n-k-1$ of $Q(n-1)$.

Now let $\cP(n,k)$ be the set of paths obtained as the union of $\cP_0\circ 0$, $\cP_1\circ 1$ plus the edges $\{(y\circ 0,y\circ 1)\mid y\in Y\}$ from the matching $M(n)$. Note that the edges added from the matching $M(n)$ extend the paths from $\cP_0$ whose end vertex is in $Y$ by one edge, so that all edges in $\cP(n,k)$ end in level $n-k$ of $Q(n)$ (see Figure~\ref{fig:ind}). Moreover, by adding the edges from the matching $M(n)$ the paths remain vertex-disjoint by the definition of $Y$.

Clearly, the number of paths in $\cP(n,k)$ is $|\cP_0|+|\cP_1|=\big(\binom{n-1}{k}-1\big)+\big(\binom{n-1}{k-1}+1\big)=\binom{n}{k}$.
Moreover, each of the paths from $\cP(n,k)$ starts at a vertex of the cycle $C(n,k)$ in level $k+1$: the vertex $a(n-1,k+1)$ in level $k+1$ of $Q(n-1)$ was removed from the cycle $C(n-1,k)$, and the corresponding path $A(n-1,k)$ was removed from $\cP(n-1,k)$ (recall the definition of $\cP_0$). On the other hand, the vertex $b(n-1,k)$ in level $k$ of $Q(n-1)$ was added to $C_1^-$, and the corresponding path $B(n-1,k-1)$ was added to $\cP(n-1,k-1)$ (recall the definition of $\cP_1$).

It remains to verify that the paths $\cP(n,k)$ satisfy conditions~(ii) and (iii).
Condition~(ii) is satisfied, as the path that contains the vertex $a(n,k+1)$ was obtained from $A(n-1,k-1)$ by removing one vertex (recall the definition of $\cP_1$), and by attaching an additional 1-bit.
Condition~(iii) is satisfied, as $A(n-1,k)$ was removed from $\cP(n-1,k)$ (recall the definition of $\cP_0$), and as we have $B(n,k)=A(n-1,k)\circ 0$.
This completes the proof.
\end{proof}

\begin{figure}
\centering
\PSforPDF{
 \psfrag{qn}{\Large $Q(n)$}
 \psfrag{kb}{\Large $k$}
 \psfrag{nmkb}{\Large $n-k$}
 \psfrag{qm0}{$Q(n-1)\circ 0$}
 \psfrag{qm1}{$Q(n-1)\circ 1$}
 \psfrag{m}{\Large $M(n)$}
 \psfrag{x}{$X$}
 \psfrag{y}{$Y$}
 \psfrag{ex}{$E_X$}
 \psfrag{k}{$k$}
 \psfrag{kp1}{$k+1$}
 \psfrag{km1}{$k-1$}
 \psfrag{nmkm1}{$n-k-1$}
 \psfrag{nmk}{$n-k$}
 \psfrag{c0}{$C(n-1,k)$}
 \psfrag{c1}{$C(n-1,k-1)$}
 \psfrag{c2}{$C(n,k)$}
 \psfrag{cp0}{$\cP(n-1,k)$}
 \psfrag{cp1}{$\cP(n-1,k-1)$}
 \psfrag{cp2}{$\cP(n,k)$}
 \psfrag{a0}{$A(n-1,k)$}
 \psfrag{a1}{$A(n-1,k-1)$}
 \psfrag{a2}{$A(n,k)$}
 \psfrag{b0}{$B(n-1,k)$}
 \psfrag{b1}{$B(n-1,k-1)$}
 \psfrag{b2}{$B(n,k)=A(n-1,k)\circ 0$}
 \psfrag{d0}{$D(n-1,k)$}
 \psfrag{d1}{$D(n-1,k-1)$}
 \psfrag{d2}{$D(n,k)$}
 \psfrag{x1}{\footnotesize $a(n-1,k)$}
 \psfrag{x2}{\footnotesize $a(n-1,k+1)$}
 \psfrag{x3}{\footnotesize $b(n-1,k)$}
 \psfrag{x4}{\footnotesize $a(n-1,k-1)$}
 \psfrag{x5}{\footnotesize $a(n-1,k)$}
 \psfrag{x6}{\footnotesize $b(n-1,k-1)$}
 \psfrag{x7}{\footnotesize $b(n-1,k)$}
 \psfrag{x8}{\footnotesize $b(n,k)$}
 \psfrag{xa}{\footnotesize $a(n,k)$}
 \psfrag{xb}{\footnotesize $a(n,k+1)$}
 \psfrag{xc}{\footnotesize $b(n,k)$}
 \psfrag{xa}{\footnotesize $a(n,k)$}
 \psfrag{xd}{\footnotesize $a(n-1,n-k-1)$}
 \psfrag{xe}{\footnotesize $a(n-1,n-k)$}
 \psfrag{xf}{\footnotesize $b(n-1,n-k-1)$}
 \includegraphics{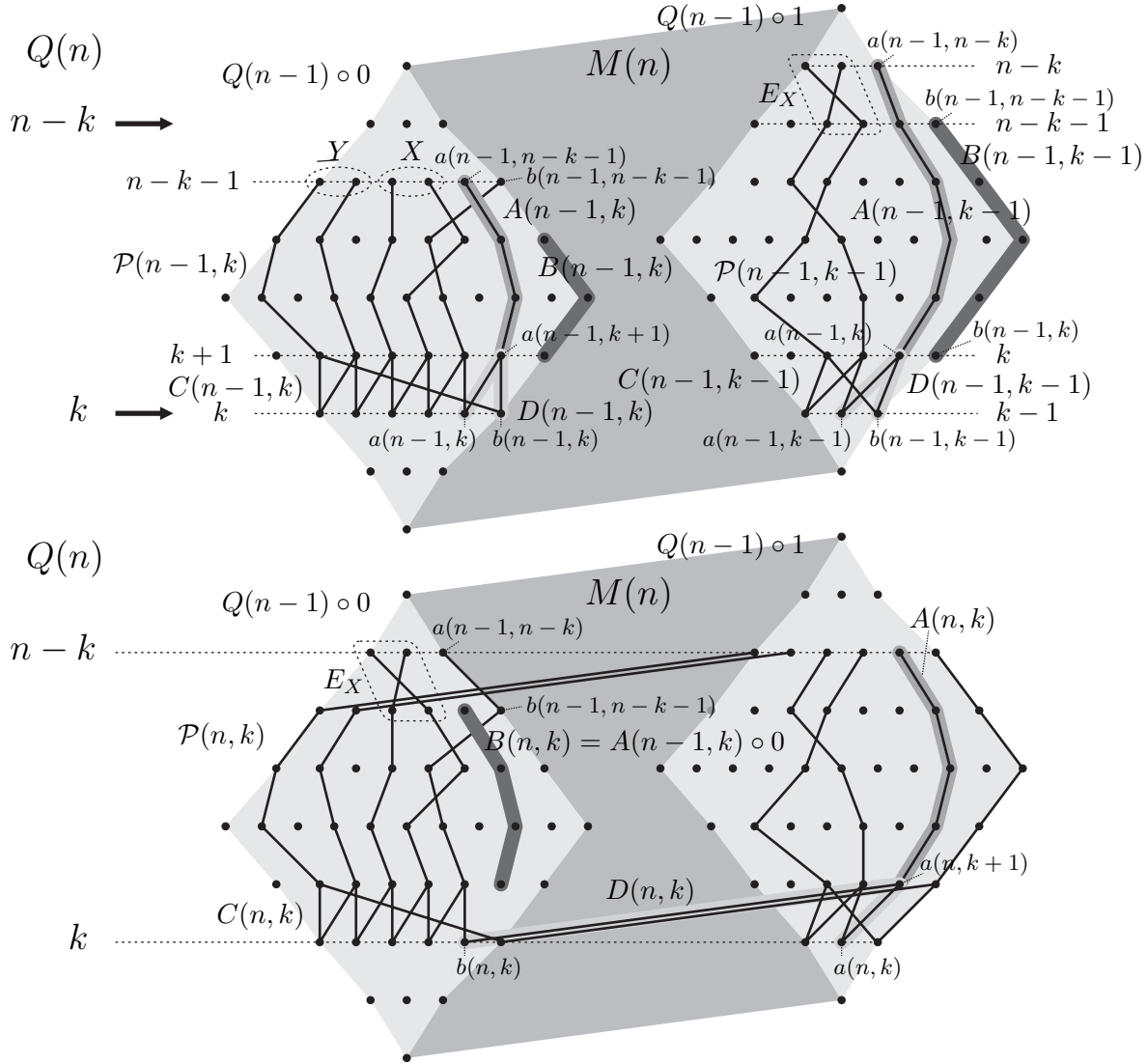}
}
\caption{Illustration of the induction step in part~(c) of the proof of Lemma~\ref{lem:aux}. The upper part shows the subgraphs $C(n-1,k)$, $\cP(n-1,k)$ and $C(n-1,k-1)$, $\cP(n-1,k-1)$ of $Q(n-1)$ used for the induction step, the lower part shows the subgraphs $C(n,k)$ and $\cP(n,k)$ of $Q(n)$ constructed from them. The edges of the matching $M(n)$ between the two copies of $Q(n-1)$ are not drawn individually, but illustrated by a dark grey region. The various paths $D(n,k)$, $A(n,k)$ and $B(n,k)$ are highlighted in light/middle/dark grey, respectively.}
\label{fig:ind}
\end{figure}

\bibliographystyle{alpha}
\bibliography{refs}

\begin{thebibliography}{HKRR05}

\bibitem[Big79]{MR556008}
N.~Biggs.
\newblock Some odd graph theory.
\newblock In {\em Second {I}nternational {C}onference on {C}ombinatorial
  {M}athematics ({N}ew {Y}ork, 1978)}, volume 319 of {\em Ann. New York Acad.
  Sci.}, pages 71--81. New York Acad. Sci., New York, 1979.

\bibitem[BW84]{MR737262}
M.~Buck and D.~Wiedemann.
\newblock Gray codes with restricted density.
\newblock {\em Discrete Math.}, 48(2-3):163--171, 1984.

\bibitem[CF02]{MR1883565}
Y.~Chen and Z.~F{\"u}redi.
\newblock Hamiltonian {K}neser graphs.
\newblock {\em Combinatorica}, 22(1):147--149, 2002.

\bibitem[Che00]{MR1778200}
Y.~Chen.
\newblock Kneser graphs are {H}amiltonian for {$n\geq 3k$}.
\newblock {\em J. Combin. Theory Ser. B}, 80(1):69--79, 2000.

\bibitem[Che03]{MR1999733}
Y.~Chen.
\newblock Triangle-free {H}amiltonian {K}neser graphs.
\newblock {\em J. Combin. Theory Ser. B}, 89(1):1--16, 2003.

\bibitem[CL87]{MR888679}
B.~Chen and K.~Lih.
\newblock Hamiltonian uniform subset graphs.
\newblock {\em J. Combin. Theory Ser. B}, 42(3):257--263, 1987.

\bibitem[DKS94]{MR1268348}
D.~Duffus, H.~Kierstead, and H.~Snevily.
\newblock An explicit {$1$}-factorization in the middle of the {B}oolean
  lattice.
\newblock {\em J. Combin. Theory Ser. A}, 65(2):334--342, 1994.

\bibitem[DSW88]{MR962223}
D.~Duffus, B.~Sands, and R.~Woodrow.
\newblock Lexicographic matchings cannot form {H}amiltonian cycles.
\newblock {\em Order}, 5(2):149--161, 1988.

\bibitem[FT95]{MR1350586}
S.~Felsner and W.~Trotter.
\newblock Colorings of diagrams of interval orders and {$\alpha$}-sequences of
  sets.
\newblock {\em Discrete Math.}, 144(1-3):23--31, 1995.
\newblock Combinatorics of ordered sets (Oberwolfach, 1991).

\bibitem[Gou91]{gould:91}
R.~Gould.
\newblock Updating the hamiltonian problem—a survey.
\newblock {\em Journal of Graph Theory}, 15(2):121--157, 1991.

\bibitem[G{\v{S}}10]{Gregor20102448}
P.~Gregor and R.~{\v{S}}krekovski.
\newblock On generalized middle-level problem.
\newblock {\em Inform. Sci.}, 180(12):2448--2457, 2010.

\bibitem[Hav83]{MR737021}
I.~Havel.
\newblock Semipaths in directed cubes.
\newblock In {\em Graphs and other combinatorial topics ({P}rague, 1982)},
  volume~59 of {\em Teubner-Texte Math.}, pages 101--108. Teubner, Leipzig,
  1983.

\bibitem[HHW88]{MR949280}
F.~Harary, J.~Hayes, and H.~Wu.
\newblock A survey of the theory of hypercube graphs.
\newblock {\em Comput. Math. Appl.}, 15(4):277--289, 1988.

\bibitem[HKRR05]{horakEtAl:05}
P.~Hor\'{a}k, T.~Kaiser, M.~Rosenfeld, and Z.~Ryj\'{a}cek.
\newblock The prism over the middle-levels graph is {H}amiltonian.
\newblock {\em Order}, 22(1):73--81, 2005.

\bibitem[Hur94]{MR1271867}
G.~Hurlbert.
\newblock The antipodal layers problem.
\newblock {\em Discrete Math.}, 128(1-3):237--245, 1994.

\bibitem[HW78]{MR510592}
K.~Heinrich and W.~Wallis.
\newblock Hamiltonian cycles in certain graphs.
\newblock {\em J. Austral. Math. Soc. Ser. A}, 26(1):89--98, 1978.

\bibitem[Joh04]{MR2046083}
R.~Johnson.
\newblock Long cycles in the middle two layers of the discrete cube.
\newblock {\em J. Combin. Theory Ser. A}, 105(2):255--271, 2004.

\bibitem[Joh11]{MR2836824}
R.~Johnson.
\newblock An inductive construction for {H}amilton cycles in {K}neser graphs.
\newblock {\em Electron. J. Combin.}, 18(1):Paper 189, 12, 2011.

\bibitem[Kar72]{Karp72}
R.~Karp.
\newblock Reducibility among combinatorial problems.
\newblock In {\em Complexity of computer computations (Proc. Sympos., IBM
  Thomas J. Watson Res. Center, Yorktown Heights, N.Y., 1972)}, pages 85--103,
  New York, 1972. Plenum.

\bibitem[KO]{kuehn:osthus:survey}
D.~Kühn and D.~Osthus.
\newblock Hamilton cycles in graphs and hypergraphs: an extremal perspective.
\newblock To appear in \textit{Proceedings of the ICM 2014}.

\bibitem[KT88]{MR962224}
H.~Kierstead and W.~Trotter.
\newblock Explicit matchings in the middle levels of the {B}oolean lattice.
\newblock {\em Order}, 5(2):163--171, 1988.

\bibitem[Lov70]{MR0263646}
L.~Lov{\'a}sz.
\newblock Problem 11, in {C}ombinatorial structures and their applications.
\newblock In {\em Proc. Calgary Internat. Conf. (Calgary, Alberta, 1969)},
  pages xvi+508, New York, 1970. Gordon and Breach Science Publishers.

\bibitem[Lov78]{MR514625}
L.~Lov{\'a}sz.
\newblock Kneser's conjecture, chromatic number, and homotopy.
\newblock {\em J. Combin. Theory Ser. A}, 25(3):319--324, 1978.

\bibitem[Mat76]{MR0389663}
M.~Mather.
\newblock The {R}ugby footballers of {C}roam.
\newblock {\em J. Combinatorial Theory Ser. B}, 20(1):62--63, 1976.

\bibitem[ML72]{MR0457282}
G.~Meredith and K.~Lloyd.
\newblock The {H}amiltonian graphs {$O_{4}$} to {$O_{7}$}.
\newblock In {\em Combinatorics ({P}roc. {C}onf. {C}ombinatorial {M}ath.,
  {M}ath. {I}nst., {O}xford, 1972)}, pages 229--236. Inst. Math. Appl.,
  Southend-on-Sea, 1972.

\bibitem[Müt14]{muetze:14}
T.~Mütze.
\newblock Proof of the middle levels conjecture.
\newblock {\it arXiv:1404.4442}, August 2014.

\bibitem[SA11]{shimada-amano}
M.~Shimada and K.~Amano.
\newblock A note on the middle levels conjecture.
\newblock {\it arXiv:0912.4564}, September 2011.

\bibitem[Sav93]{savage:93}
C.~Savage.
\newblock Long cycles in the middle two levels of the boolean lattice.
\newblock {\em Ars Combin.}, 35-A:97--108, 1993.

\bibitem[Sav97]{Savage:1997}
C.~Savage.
\newblock A survey of combinatorial {G}ray codes.
\newblock {\em SIAM Rev.}, 39(4):605--629, 1997.

\bibitem[Sim91]{MR1152123}
J.~Simpson.
\newblock Hamiltonian bipartite graphs.
\newblock In {\em Proceedings of the {T}wenty-second {S}outheastern
  {C}onference on {C}ombinatorics, {G}raph {T}heory, and {C}omputing ({B}aton
  {R}ouge, {LA}, 1991)}, volume~85, pages 97--110, 1991.

\bibitem[Sim94]{MR1282567}
J.~Simpson.
\newblock On uniform subset graphs.
\newblock {\em Ars Combin.}, 37:309--318, 1994.

\bibitem[SS04]{MR2020936}
I.~Shields and C.~Savage.
\newblock A note on {H}amilton cycles in {K}neser graphs.
\newblock {\em Bull. Inst. Combin. Appl.}, 40:13--22, 2004.

\bibitem[SSS09]{MR2548541}
I.~Shields, B.~Shields, and C.~Savage.
\newblock An update on the middle levels problem.
\newblock {\em Discrete Math.}, 309(17):5271--5277, 2009.

\bibitem[SW95]{MR1329390}
C.~Savage and P.~Winkler.
\newblock Monotone {G}ray codes and the middle levels problem.
\newblock {\em J. Combin. Theory Ser. A}, 70(2):230--248, 1995.

\end{thebibliography}

\end{document}